\documentclass[11pt,leqno]{amsart}
\usepackage{amsmath,amssymb,amsthm}
\usepackage{hyperref,tikz-cd,graphicx}

\usepackage{tikz}
\usetikzlibrary{arrows}
\usetikzlibrary{decorations.markings}
\tikzset{nimplies/.style={
		decoration={markings,
			mark= at position 0.5 with {
				\node[transform shape] (tempnode){\tiny/};
			}},postaction={decorate}}}
			
\tikzset{negated/.style={
        decoration={markings,
            mark= at position 0.5 with {
                \node[transform shape] (tempnode) {$\backslash$};
            }
        },
        postaction={decorate}
    }
}
\usepackage{float}
\usepackage{xcolor}
\usepackage[shortlabels]{enumitem}

\newtheorem{theorem}{Theorem}[section]
\newtheorem*{theorem*}{Theorem}
\newtheorem{lemma}[theorem]{Lemma}
\newtheorem{ques}[theorem]{Question}

\newtheorem{proposition}[theorem]{Proposition}
\newtheorem{corollary}[theorem]{Corollary}

\theoremstyle{definition}
\newtheorem{definition}[theorem]{Definition}
\newtheorem{example}[theorem]{Example}

\theoremstyle{remark}
\newtheorem{remark}[theorem]{Remark}
\numberwithin{equation}{section}

\def\ignora#1{}

\def\nn#1{\left\vert  \! \left\vert \! \left\vert \, #1 \, \right\vert \!
	\right\vert \! \right\vert }

\DeclareMathOperator{\dens}{dens}

\renewcommand{\geq}{\geqslant}
\renewcommand{\leq}{\leqslant}

\newcommand{\norm}[1]{\left\Vert#1\right\Vert}

\newcommand{\spn}{\operatorname{span}}

\newcommand{\Xs}{X^{\ast}}

\newcommand{\Xss}{X^{\ast\ast}}

\newcommand{\supp}{\operatorname{supp}}

\newcommand{\e}{\varepsilon}
\newcommand{\cl}{\mathcal}
\newcommand{\n}{\norm}
\newcommand{\bb}{\mathbb}

\begin{document}
	
\title[A Characterization of Banach Spaces Containing $\ell_1(\kappa)$]{A Characterization of Banach Spaces Containing $\ell_1(\kappa)$ via ball-covering properties}
\author {Stefano Ciaci, Johann Langemets, Aleksei Lissitsin} 
\address{Institute of Mathematics and Statistics, University of Tartu, Narva mnt 18, 51009 Tartu, Estonia}
\email{stefano.ciaci@ut.ee, johann.langemets@ut.ee, aleksei.lissitsin@ut.ee}
\thanks{This work was supported by the Estonian Research Council grants (PSG487) and (PRG877).}
\urladdr{\url{https://johannlangemets.wordpress.com/}}

\begin{abstract}
	In 1989, G.~Godefroy proved that a Banach space contains an isomorphic copy of $\ell_1$ if and only if it can be equivalently renormed to be octahedral. It is known that octahedral norms can be characterized by means of covering the unit sphere by a finite number of balls. This observation allows us to connect the theory of octahedral norms with ball-covering properties of Banach spaces introduced by L.~Cheng in 2006. Following this idea, we extend G.~Godefroy's result to higher cardinalities. We prove that, for an infinite cardinal $\kappa$, a Banach space $X$ contains an isomorphic copy of $\ell_1(\kappa^+)$ if and only if it can be equivalently renormed in such a way that its unit sphere cannot be covered by $\kappa$ many open balls not containing $\alpha B_X$, where $\alpha\in (0,1)$. We also investigate the relation between ball-coverings of the unit sphere and octahedral norms in the setting of higher cardinalities.
\end{abstract}

\maketitle

\section{Introduction}\label{sec: intro}

The existence of isomorphic copies of $\ell_1$ in a Banach space has been the interest of many mathematicians. Probably one of the most known results in this direction is the one given by H.~Rosenthal in \cite{Rosenthal}, which is independent of the norm considered in the space. There are also purely geometrical characterizations in terms of the considered norm in the space. For this reason octahedral norms were introduced in an unpublished paper by G.~Godefroy and B.~Maurey (see \cite{godefroy_metric_1989} and \cite{godefroy_maurey}).

We recall that the norm $\|\cdot\|$ on a normed space $X$ (or $X$) is called \emph{octahedral}
if, for every finite-dimensional subspace $E$ of $X$ and $\varepsilon>\nobreak0$,
there is $y\in S_X$ such that
\[
\|x+y\|\geq(1-\varepsilon)\bigl(\|x\|+\|y\|\bigr)\quad\text{for all $x\in E$.}
\]

If $X$ is a separable Banach space, then G.~Godefroy and N.~Kalton (see \cite[Lemma~9.1]{GodefroyBallTopology}) proved that the octahedrality of $X$ is equivalent to the existence of some $x^{\ast\ast}\in X^{\ast\ast}\setminus\{0\}$ such that 
\[
\|x+x^{\ast\ast}\|=\|x\|+\|x^{\ast\ast}\| \quad \text{for all $x\in X$}.
\]

Therefore, it is clear that $\ell_1$ (also $L_1[0,1]$ and $C([0,1])$), being a separable Banach space, is octahedral. If we allow the space to be renormed, then G. Godefroy proved the following general characterization:

\begin{theorem}[see {\cite[Theorem~II.4]{godefroy_metric_1989}}]\label{thm: godefroy contain ell1} Let $X$ be a Banach space. Then $X$ contains an isomorphic copy of $\ell_1$ if and only if there exists an equivalent norm $\nn\cdot$ in $X$ such that $(X,\nn\cdot)$ is octahedral.
\end{theorem}
 
It is known that a Banach space $X$ is octahedral whenever $\Xss$ is. The converse is not true in general. The natural norm of $C([0,1])$ is octahedral, but its bidual norm is not. However, if $X$ is a separable Banach space containing $\ell_1$, then there always exists an equivalent norm on $X$ such that the bidual $\Xss$ is octahedral (see \cite{langemets_bidual_2019}). The non-separable case remains unknown as far as the authors know.

G.~Godefroy also remarked (without a proof) in \cite[p.~12]{godefroy_metric_1989} that the norm on a Banach space $X$ is octahedral if and only if every finite covering of $B_X$, the closed unit ball of $X$, with closed balls has at least one member of the covering that itself contains $B_X$. A proof of this fact can be found in \cite{HallerGeometry2014}. 

Thus, octahedral norms are closely connected to coverings of the unit ball or the unit sphere in a Banach space, which is a frequent topic in the literature (see \cite{PapiniCovering2009} for a nice survey). In this direction L.~Cheng introduced the ball-covering property (see \cite{Cheng2006}). We recall that a normed space is said to have the \emph{ball-covering property} (BCP, for short) if its unit sphere can be covered by the union of countably many closed balls not containing the origin. It is known that separable Banach spaces have the BCP, but the converse is not true, because $\ell_\infty$ has the BCP. In recent years the BCP of Banach spaces has been studied by various authors (see e.g., \cite{ChengNote2010}, \cite{ChengRenorming2009}, \cite{fonf-zanco}, \cite{Guirao2019} and \cite{Luo2020}). 

Following \cite{ChengMinimal2009} and clarifying the state-of-the-art on the BCP, A.~J.~Guirao, A.~Lissitsin and V.~Montesinos (see \cite{Guirao2019}) extended the original BCP to the $\alpha$-BCP, where $\alpha\in [-1,1)$. A normed space $X$ is said to have the \emph{$\alpha$-BCP}, for $\alpha\in[0,1)$, if its unit sphere can be covered by countably many balls that do not intersect $\alpha B_X$, while, for $\alpha\in (-1,0)$, the balls are required not to contain $\alpha B_X$. In this language the original BCP corresponds to the $0$-BCP.

In this paper we consider a slight generalisation of the $\alpha$-BCP, namely the $\alpha$-BCP$_\kappa$ (see Definition~\ref{def: alpha-BCP_kappa}), in which we substitute the condition of the covering from being countable to being of cardinality $\kappa$, where $\kappa$ is some infinite cardinal. In addition, we consider the property $\alpha$-BCP$_{<\omega}$ (see Definition~\ref{def: alpha-BCP_finite}), in which we require the covering to be finite. Our starting point is the observation that a Banach space fails the $\alpha$-BCP$_{<\omega}$ for all $\alpha\in (-1,1)$ if and only if it is octahedral (see Proposition~\ref{prop: finite BCP and OH}). Therefore, G.~Godefroy's result (see Theorem~\ref{thm: godefroy contain ell1}) can be restated in terms of failing the $\alpha$-BCP$_{<\omega}$ (see Remark~\ref{rem: rephrase Godefroy's theorem}). Our main result in this paper is an extension of Theorem~\ref{thm: godefroy contain ell1} to infinite cardinals.

\begin{theorem*}
	Let $X$ be a Banach space and $\kappa$ an infinite cardinal. Then $X$ contains an isomorphic copy of $\ell_1(\kappa^+)$ if and only if it can be equivalently renormed to fail the $\alpha$-BCP$_{\kappa}$ for every $($or some$)$ $\alpha\in (-1,0)$.
\end{theorem*}

Let us now describe the organization of the paper. In Section~\ref{sec: preliminaries} we introduce some notation for coverings and collect some elementary properties of them, which will be used throughout the paper. Also, we define the $\alpha$-BCP$_\kappa$ and show that a Banach space fails the $\alpha$-BCP$_{<\omega}$ for all $\alpha\in (-1,1)$ if and only if it is octahedral (see Proposition~\ref{prop: finite BCP and OH}). Section~3 is purely devoted to proving the promised characterization of Banach spaces containing $\ell_1(\kappa)$ (see Theorem~\ref{thm: renorming ell_1(kappa)}). In Section~\ref{sec: direct sums} we investigate the stability of coverings in direct sums of Banach spaces. On one hand these results complement the existing ones from the literature. On the other hand, we will use them in Section~\ref{sec: examples} to construct various spaces with and without the $\alpha$-BCP$_\kappa$. We begin Section~\ref{sec: examples} by proving that in Lipschitz spaces the notions of octahedral norm and failure of the $(-1)$-BCP coincide (see Proposition~\ref{prop: OH and failure of (-1)-BCP are the same in Lip(M)}). In the remaining part of Section~\ref{sec: examples} we clarify the interrelations of different versions of the $\alpha$-BCP$_\kappa$ and octahedral norms (see Figure~\ref{fig: diagram}). Namely, we prove that there exists a Banach space, which fails the $\alpha$-BCP$_{\kappa}$ for every $\alpha\in(-1,1)$, but that has the $(-1)$-BCP (see Theorem~\ref{thm: X fails (-alpha)-BCP for all alpha in (0,1)}). We also introduce the natural notion of a $\kappa$-octahedral norm (see Definition~\ref{def: kappa-octahedrality}) and show that there exists a Banach space which is $\kappa$-octahedral, but that has the $\alpha$-BCP for all $\alpha\in [-1,0)$. Finally, in Section~\ref{sec: questions} we point out some open questions that are suggested by the current work. 

We pass now to introduce some notation. We consider only real Banach spaces. For a Banach space $X$, $\Xs$ denotes its topological dual and $B(x,r)$ the open ball centered in $x$ of radius $r$. By $B_X$ we denote the closed unit ball and by $S_X$ the unit sphere of $X$. By dens$(E)$ we represent the density of a topological space $E$ and by $\supp(f)$ the support of a function $f$. For a cardinal $\kappa$, by cf$(\kappa)$ we denote its cofinality and by $\kappa^+$ its successor cardinal.

Given a metric space $M$ with a designated origin $0$, we will
denote by $\text{Lip}_0(M)$ the Banach space of all real-valued Lipschitz functions on $M$ which vanish at $0$. 

\section{Preliminaries}\label{sec: preliminaries}
    
	Let $X$ be a normed space, $A\subset X$, $\kappa$ a cardinal and $\alpha,\beta\in [-1,1]$. Define
	$$C(\alpha,\beta,A):=\{x\in X:\exists a\in A\hspace{1mm}(\|x-a\|<\alpha\|x\|+\beta\|a\|)\}$$
	and
	$$C(\alpha,\beta,\kappa):=\{B\subset C(\alpha,\beta,A):\text{ for some set }A\text{ of cardinality }\kappa\}.$$
	Notice that this notation can represent coverings of the sphere since the condition $S_X\in C(\alpha,\beta,\kappa)$ holds if and only if there exists a set $A\subset X$ of cardinality $\kappa$ such that
	$$S_X\subset\bigcup_{a\in A}B(a,\alpha+\beta\|a\|).$$
	
	Now we list some elementary but useful properties regarding the set $C(\alpha,\beta,A)$.
	
	\begin{remark}\label{rem: properties of coverings}
	Let $X$ be a normed space, $\alpha,\beta\in [-1,1]$ and $A,B\subset X$. Then
		\begin{enumerate}
		\item[$(a)$] $C(\alpha,\beta,A)=C(\alpha,\beta,\overline{A})$.
		\item[$(b)$] If $B\subset C(\alpha,\beta,A)$, then $tB\subset C(\alpha,\beta,tA)$ for all $t>0$.
		\item[$(c)$] If $S_X\subset C(\alpha,\beta,A)$, then $X\setminus\{0\}=C(\alpha,\beta,\mathbb{R}^+A)$.
		\item[$(d)$] Let $\{t_a:a\in A\}\subset [1,\infty)$. If $B\subset C(\alpha,1,A)$, then
	$$B\subset C(\alpha,1,\{t_aa:a\in A\}).$$
	\end{enumerate}
	
\end{remark}

\begin{proposition}\label{prop: equivalent conditions coverings}
	Let $X$ be a normed space, $\kappa$ an infinite cardinal and $\alpha,\beta\in[-1,1]$. Then the following are equivalent:
	\begin{itemize}
	    \item[$(i)$] $S_X\notin C(\alpha,\beta,\kappa)$.
	    \item[$(ii)$] $X\setminus\{0\}\notin C(\alpha,\beta,\kappa)$.
		\item[$(iii)$] For all subspaces $Y\subset X$ with $\dens(Y)\le\kappa$ there exists $x\in S_X$ such that for all $\lambda\in\mathbb{R}$ and $y\in Y$ we have $\|\lambda x+y\|\ge\alpha|\lambda|+\beta\|y\|$.
	\end{itemize}
\end{proposition}
\begin{proof}
    $(ii)\implies(i)$ follows from Remark~\ref{rem: properties of coverings}(a) and (c).
    
    $(i)\implies (iii)$. Assume that $Y=\overline{A}$ for some set $A\subset X$ of cardinality $\kappa$. Remark~\ref{rem: properties of coverings}(a) implies that $S_X\not\subset C(\alpha,\beta,A)=C(\alpha,\beta,Y)$. Therefore there exists $x\in S_X$ such that $\|x+y\|\ge \alpha+\beta\|y\|$ for all $y\in Y$, thus for all $\lambda\in\mathbb{R}\setminus\{0\}$ and $y\in Y$
    $$\|\lambda x+y\|=|\lambda|\|x+y/\lambda\|\ge|\lambda|(\alpha+\beta\|y/\lambda\|)=\alpha|\lambda|+\beta\|y\|.$$
    $(iii)\implies (ii)$ is trivial.
\end{proof}

We are now ready to extend the $\alpha$-BCP to bigger cardinals.

\begin{definition}\label{def: alpha-BCP_kappa}
	Let $X$ be a normed space, $\kappa$ an infinite cardinal and $\alpha\in [-1,1)$. We say that $X$ has the \emph{$\alpha$-BCP$_{\kappa}$} if $S_X\in C(-\alpha,1,\kappa)$.
\end{definition}

In this notation the usual $\alpha$-BCP corresponds to the $\alpha$-BCP$_{\omega}$. By Proposition~\ref{prop: equivalent conditions coverings}, a normed space $X$ has the $\alpha$-BCP$_\kappa$ if and only if $X\setminus\{0\}\in C(-\alpha,1,\kappa)$. Moreover, as already noted earlier, $X$ has the $\alpha$-BCP$_\kappa$ if and only if there exists a set $A\subset X$ of cardinality $\kappa$ such that
$$S_X\subset\bigcup_{a\in A} B(a,\|a\|-\alpha).$$

In addition, let us now consider the following finite version of the BCP.

\begin{definition}\label{def: alpha-BCP_finite}
    Let $X$ be a normed space and $\alpha\in [-1,1)$. We say that $X$ has the \emph{$\alpha$-BCP$_{<\omega}$} if there exists a finite set $A\subset X$ such that $S_X\in C(-\alpha,1,A)$.
\end{definition}

We observe now that octahedral norms can be characterized in terms of failing the $\alpha$-BCP$_{<\omega}$.

\begin{proposition}\label{prop: finite BCP and OH}
    Let $X$ be a normed space. Then the following are equivalent:
    \begin{itemize}
        \item [$(i)$] $X$ is octahedral.
        \item [$(ii)$] For all $n\in\mathbb{N}$, $a_1,\ldots,a_n\in S_X$ and $\varepsilon>0$ there exists $x\in S_X$ such that
        $$\|x-a_j\|\ge 2-\varepsilon\hspace{2mm}\text{ for all $j=1,\ldots,n$.}$$
        \item [$(iii)$] $X$ fails the $\alpha$-BCP$_{<\omega}$ for all $\alpha\in(-1,1)$.
    \end{itemize}
\end{proposition}

\begin{proof}
    $(i)\iff (ii)$ is known from \cite[Proposition 2.2]{HLP}.
    
     $(i)\implies(iii)$. Let $A\subset X$ be a finite set and $\alpha\in (-1,1)$. Find $\beta\in(0,1)$ such that $\beta\ge(-\alpha+\|a\|)/(1+\|a\|)$ for all $a\in A$. Since $X$ has an octahedral norm, then there exists $x\in S_X$ such that 
    $$\|x-a\|\ge\beta(1+\|a\|)\ge-\alpha+\|a\|$$
    for all $a\in A$. Thus, $X$ fails the $\alpha$-BCP$_{<\omega}$ for all $\alpha\in(-1,1)$.
    
    $(iii)\implies(ii)$ is obvious.
\end{proof}

\begin{remark}\label{rem: rephrase Godefroy's theorem}
Using Proposition~\ref{prop: finite BCP and OH}, we can now rephrase Theorem~\ref{thm: godefroy contain ell1} as follows. 

Let $X$ be a Banach space. Then $X$ contains an isomorphic copy of $\ell_1$ if and only if it can be equivalently renormed to fail the $\alpha$-BCP$_{<\omega}$ for all $\alpha\in(-1,1)$.
\end{remark}

\section{Characterizing Banach spaces containing $\ell_1(\kappa)$}\label{sec: main result}

We can now state and prove the main result of this paper.

\begin{theorem}\label{thm: renorming ell_1(kappa)}
	Let $X$ be a Banach space and $\kappa$ an infinite cardinal. Then $X$ contains an isomorphic copy of $\ell_1(\kappa^+)$ if and only if it can be equivalently renormed to fail the $\alpha$-BCP$_{\kappa}$ for every $($or some$)$ $\alpha\in (-1,0)$.
\end{theorem}

The proof follows immediately from Propositions~\ref{prop: X contains ell_1} and \ref{prop: renorm X to fail -alpha-BCP} combined with the fact that cf$(\kappa^+)=\kappa^+$ for an infinite cardinal $\kappa$.

In \cite[Proposition~23]{Guirao2019} it was shown that if $X$ fails the $(-1)$-BCP, then $X$ contains $\ell_1(\Gamma)$ for some uncountable set $\Gamma$. Using essentially the same scheme, we generalize this result to arbitrary infinite cardinals, while (strictly) weakening the assumption on $X$. Indeed, in Theorem~\ref{thm: X fails (-alpha)-BCP for all alpha in (0,1)}, we will prove, for all infinite cardinals $\kappa$, the existence of a Banach space failing the $\alpha$-BCP$_\kappa$ for all $\alpha\in (-1,1)$, but having the $(-1)$-BCP.

\begin{proposition}\label{prop: X contains ell_1}
	Let $X$ be a Banach space, $\alpha \in [-1, 0)$ and $\kappa$ an infinite cardinal. If $X$ fails the $\alpha$-BCP$_{\kappa}$, then it contains an isomorphic copy of $\ell_1(\kappa^+)$.
\end{proposition}
\begin{proof}
We will follow the ideas of \cite[Proposition 23]{Guirao2019}. It suffices to show that there is a set $A\subset S_X$ of cardinality $\kappa^+$ such that
	$$\|\sum_{i=1}^n\lambda_ia_i\|\ge-\alpha\sum_{i=1}^n|\lambda_i|$$
	for all $n\in\mathbb{N}$, $\{a_1, \dots, a_n\}\subset A$ and $\{\lambda_1, \dots, \lambda_n\} \subset \mathbb R$. 
	
	Call any set (not necessarily of cardinality $\kappa^+$) satisfying this property \emph{good} and let $\cl P$ be the partially ordered set of all good sets in $S_X$.
	
	Notice that $\cl P\not=\emptyset$ since $\{x\} \in \cl P$ for any $x \in S_X$. Moreover, for any chain $(A_{\eta}) \subset \cl P$, we have that $\bigcup_{\eta} A_{\eta}\in\cl P$ is an upper bound. By Zorn's Lemma $\cl P$ contains a maximal element $A$. 
	
	We claim that $|A| \geq \kappa^+$. Assume by contradiction otherwise. Since $X$ fails the $\alpha$-BCP$_{\kappa}$, by Proposition~\ref{prop: equivalent conditions coverings}, we can find $x\in S_X$ such that
	$\n{\lambda x+a} \ge-\alpha |\lambda|+\|a\|$ for all $a\in\spn(A)$ and $\lambda \in \mathbb R$. In particular
	$$\|\lambda x+\sum_{i=1}^n\lambda_ia_i\|\ge
	-\alpha\big(|\lambda|+\sum_{i=1}^n|\lambda_i|\big)$$
	holds for all $n\in\mathbb{N}$, $\{a_1, \dots, a_n\} \subset A$ and $\{\lambda, \lambda_1, \dots, \lambda_n\} \subset \mathbb R$. This means that $A\cup \{x\} \in \cl P$, which contradicts the maximality of $A$.
\end{proof}

The second ingredient in the proof of Theorem~\ref{thm: renorming ell_1(kappa)} is the following result. The cornerstone of its proof is the norm construction technique from \cite[Section 4]{Kadets2011}.

\begin{proposition}\label{prop: renorm X to fail -alpha-BCP}
	Let $X$ be a Banach space and $\kappa$ an infinite cardinal. If $\ell_1(\kappa)\subset X$ isomorphically, then there exists an equivalent norm $\nn{\cdot}$ on $X$ such that $(X,\nn{\cdot})$ fails the $\alpha$-BCP$_{\eta}$ for all $\alpha\in(-1,1)$ and $\eta<\text{cf}(\kappa)$.
\end{proposition}

The proof of Proposition~\ref{prop: renorm X to fail -alpha-BCP} makes use of the following lemma.

\begin{lemma}\label{lem: equivalent norm}
	Let $X$ be a normed space and $Y$ a closed subspace of $X$. If $p$ is a seminorm on $X$ dominated by the norm of $X$ and equivalent to it on $Y$, then 
	\[
	\nn{x}:= p(x) + \n{x+Y}_{X/Y}
	\]
	defines an equivalent norm on $X$. 
	
	Furthermore, if $\kappa$ is an infinite cardinal, $\alpha\in [-1,1)$ and for every set $A\subset X$ of cardinality $\kappa$ there exists $y\in Y\setminus\{0\}$ such that 
	\[p(y-a)\ge-\alpha p(y)+p(a)\quad \text{ for all $a\in A$,}
	\]
	then $(X,\nn{\cdot})$ fails the $\alpha$-BCP$_{\kappa}$.
	\end{lemma}
\begin{proof}
Assume $p \leq C \n{\cdot}$ on $X$ and $p \geq c \n \cdot$ on $Y$ for some $C,c > 0$.
On one hand $\nn{\cdot} \leq (C+1) \n \cdot$. On the other hand we claim that $\nn\cdot\ge\tilde{c}\|\cdot\|$, where $\tilde{c}:=c/(1+c+C)$. 

Fix $x\in X$. Without loss of generality we can assume that $\n{x+Y} < \tilde{c}\n x$, otherwise the claim is trivial. There is $y \in Y$
such that $\n{x-y} < \tilde{c} \n{x}$ and hence $\n y \geq \n x - \n{x-y} > (1-\tilde{c}) \n x$. Therefore,
$$\nn{x} \geq p(x) \geq p(y) - p(x-y) \geq c\n y - C\n{x-y} >$$
$$>c(1-\tilde{c})\n x - C\tilde{c}\n x=\tilde{c}\|x\|. $$
Thus the claim is proved. For the furthermore part fix a set $A\subset X$ of cardinality $\kappa$. There is $y\in Y\setminus\{0\}$ such that $p(y-a)\ge-\alpha p(y)+p(a)$ for all $a\in A$. Thus,
$$\nn{y-a}=p(y-a)+\|a+Y\|_{X/Y}\ge-\alpha p(y)+\nn a=-\alpha\nn y+\nn a.$$
\end{proof}

We are now ready to prove Proposition~\ref{prop: renorm X to fail -alpha-BCP}. 

\begin{proof}[Proof of Proposition~\ref{prop: renorm X to fail -alpha-BCP}]
	If $\ell_1(\kappa)\subset X$ isomorphically, then we can renorm $X$ such that $Y:=\ell_1(\kappa)\subset X$ isometrically. Let $\{e_{\eta}:\eta<\kappa\}\subset S_Y$ satisfy
	$$\|\sum_{j=1}^n \lambda_j e_{\eta_j}\| = \sum_{j=1}^n |\lambda_j|$$
    for all $n \in \bb N$, $\{\lambda_1, \dots, \lambda_n\} \subset \bb R$ and $\{\eta_1, \dots, \eta_n\} \subset \kappa$. Let $\mathcal{P}$ be the family of all seminorms on $X$ dominated by its norm such that they coincide with it on $Y$. Notice that $\mathcal{P}\not=\emptyset$ since $\|\cdot\|\in\mathcal{P}$. Let $p:=\inf\mathcal{P}$ and observe that $p\in\mathcal{P}$ is a well defined seminorm. Applying the claim in the proof of \cite[Theorem~1.3]{BosenkoKadets2010} to ultrafilters in $\kappa$, we obtain that
	$$\lim_{\eta}p(x+e_{\eta})=p(x)+1\text{ for all }x\in X.$$
	Now fix $\alpha\in (-1,1)$ and a set $A\subset X$ of cardinality strictly smaller then cf$(\kappa)$. For all $a\in A$ there exists $\eta_a<\kappa$ such that for all $\eta_a<\eta<\kappa$
	$$p(a+e_{\eta})\ge p(a)-\alpha.$$
	Since $|A|<\text{cf}(\kappa)$, there exists some $\theta<\kappa$ such that for all $a\in A$ we have that $\eta_a<\theta$ and of course
	$$p(a+e_\theta)\ge p(a)-\alpha.$$
	Finally, we apply Lemma~\ref{lem: equivalent norm} to obtain that $(X,\nn{\cdot})$ fails the $\alpha$-BCP$_{\eta}$ whenever $\eta<\text{cf}(\kappa)$.
\end{proof}

\section{Ball-covering properties in direct sums}\label{sec: direct sums}

In \cite[Proposition~8]{Guirao2019} the authors investigated the stability of the $\alpha$-BCP in $\ell_p$-sums of normed spaces. They proved that if $X$ and $Y$ have the $\alpha$-BCP, then so does $X\oplus_p Y$ for any $1\leq p\leq \infty$. An inspection of their proof yields that a similar statement also holds for the $\alpha$-BCP$_\kappa$. Our first aim in this section is to investigate the converses to the aforementioned results and for that we will use the notion of an absolute sum.

Recall that a norm $N$ on $\mathbb R^2$ is called \emph{absolute} (see \cite{Bonsall1973}) if
\[
N(a,b)=N(|a|,|b|)\qquad\text{for all $(a,b)\in\mathbb R^2$}
\]
and \emph{normalized} if $N(1,0)=N(0,1)=1$.

For example, the $\ell_p$-norm on $\mathbb{R}^2$ is absolute and normalized for every $p\in[1,\infty]$. If $N$ is an absolute norm on $\mathbb R^2$, then $N(a,b)\leq N(c,d)$ whenever $|a|\leq |c|$ and $|b|\leq |d|$.

If $X$ and $Y$ are normed spaces and $N$ is an absolute normalized norm on $\mathbb R^2$, then we denote by $X\oplus_N Y$ the product space $X\times Y$ endowed with the norm $\|(x,y)\|_N:=N(\|x\|,\|y\|)$.

\begin{proposition}\label{prop: absolute sum}
	Let $X$ and $Y$ be normed spaces, $N$ an absolute normalized norm on $\mathbb{R}^2$, $\alpha\in[-1,0]$ and $\beta\in[-1,1]$. If $S_{X\oplus_N Y}\subset C(\alpha,\beta,A)$ for some set $A\subset X\oplus_NY$, then $S_X\subset C(\alpha,\beta,\pi_X(A))$, where $\pi_X$ is the projection from $X\oplus_N Y$ onto $X$. Moreover, if $N$ is the $\ell_1$ norm on $\mathbb{R}^2$, then the statement holds also for $\alpha\in(0,1)$.
\end{proposition}
\begin{proof}
	Fix $x\in S_X$ and consider the element $(x,0)\in S_{X\oplus_NY}$. By assumption there is $(a_X,a_Y)\in A$ such that
	$$N(\|x-a_X\|,\|a_Y\|)<\alpha+\beta N(\|a_X\|,\|a_Y\|).$$
	Assume by contradiction that $\|x-a_X\|\ge\alpha+\beta\|a_X\|$. Therefore,
	$$\alpha+\beta N(\|a_X\|,\|a_Y\|)>N(\|x-a_X\|,\|a_Y\|)\ge N(\alpha+\beta\|a_X\|,\|a_Y\|)\ge$$
	$$\ge -N(\alpha,0)+N(\beta\|a_X\|,\|a_Y\|)=-|\alpha|+N(\beta\|a_X\|,\|a_Y\|).$$
	Since $\alpha\in [-1,0]$ we have that $-|\alpha|=\alpha$. In addition notice that $$N(\beta\|a_X\|,\|a_Y\|)\ge\beta N(\|a_X\|,\|a_Y\|),$$
	which leads to a contradiction.
	
	For the moreover part, if $N$ is the $\ell_1$ norm on $\mathbb{R}^2$, we only need to adjust the previous inequality to
	$$\alpha+\beta(\|a_X\|+\|a_Y\|)>\|x-a_X\|+\|a_Y\|\ge\alpha+\beta\|a_X\|+\|a_Y\|$$
	to get a contradiction.
\end{proof}

\begin{corollary}\label{cor: absolute sum}
	Let $X$ and $Y$ be normed spaces, $N$ an absolute normalized norm in $\mathbb{R}^2$, $\kappa$ an infinite cardinal and $\alpha\in[0,1)$. If $X\oplus_N Y$ has the $\alpha$-BCP$_{\kappa}$, then $X$ and $Y$ have the $\alpha$-BCP$_{\kappa}$.
\end{corollary}

We now turn our attention to the behaviour of the $\alpha$-BCP$_{\kappa}$ in infinite direct sums of normed spaces. Firstly, recall that for a sequence $(X_n)$ of normed spaces and $1\leq p\leq\infty$, the normed space $\ell_p(X_n)$ has the BCP if and only if each $X_n$ has the BCP (see \cite[Corollary~2.7 and Theorem~2.8]{Luo2020}). Our goal in the remaining part of this section is to investigate the $\alpha$-BCP$_{\kappa}$ in infinite $\ell_1$ and $\ell_\infty$-sums in more detail.

\begin{proposition}\label{prop: infty sums}
		Let $(X_\eta)$ be a family of normed spaces and for each $\eta$ let $A_\eta$ be a subset of $X_\eta$. For every $\theta$ consider the map $\tilde{(\cdot)}:X_\theta\rightarrow\ell_{\infty}(X_\eta)$ defined by $\tilde{x}=(\ldots,0,x,0,\ldots)$ for all $x\in X_\theta$.
	\begin{itemize}
		\item[$(a)$] If $\alpha\in [0,1]$ and $S_{X_\eta}\subset C(\alpha,1,A_\eta)$ for all $\eta$, then $S_{\ell_{\infty}(X_\eta)}\subset C(\alpha,1,\mathbb{R}^+\bigcup_\eta\tilde{A_\eta})$.
		\item[$(b)$] If $\alpha\in (-1,0)$ and $S_{X_\eta}\subset C(\alpha,1,A_\eta)$ for all $\eta$, then $S_{\ell_{\infty}(X_\eta)}\subset C(\alpha+\varepsilon,1,\mathbb{R}^+\bigcup_\eta\tilde{A}_\eta)$ for every $\varepsilon>0$.
	\end{itemize}
\end{proposition}
\begin{proof}
	Assume that $S_{X_\eta}\subset C(\alpha,1,A_\eta)$ for all $\eta$. Then, by Remark~\ref{rem: properties of coverings}(c), one has that $X_\eta\setminus\{0\}=C(\alpha,1,\mathbb{R}^+A_\eta)$. Fix $x\in S_{\ell_{\infty}(X_\eta)}$.
	\begin{itemize}
		\item[$(a)$] If $\alpha\in[0,1]$, then find some $\eta$ such that $\|x(\eta)\|>0$.
		\item[$(b)$] If $\alpha\in(-1,0)$, then fix $\varepsilon\in (0,-\alpha)$ and find $\eta$ such that $\|x(\eta)\|\ge 1+\varepsilon/\alpha$.
	\end{itemize}
	By our assumption there is some $a\in \mathbb{R}^+A_\eta\setminus\{0\}$ such that $\|x(\eta)-a\|<\alpha\|x(\eta)\|+\|a\|$. Hence, 
	$$\|x(\eta)-ta\|<\alpha\|x(\eta)\|+\|ta\|$$
	for any $t\ge\max\{2/\|a\|,1\}$ by Remark~\ref{rem: properties of coverings}(d). Notice that with this choice of $t$ we have that $\|ta\|\ge2$, therefore $\|x-t\tilde{a}\|_{\infty}=\|x(\eta)-ta\|$. We deduce that
	$$\|x-t\tilde{a}\|_{\infty}=\|x(\eta)-ta\|<\alpha\|x(\eta)\|+\|ta\|=\alpha\|x(\eta)\|+\|t\tilde{a}\|_{\infty}.$$
	\begin{itemize}
		\item[$(a)$]If $\alpha\in[0,1]$, then $\alpha\|x(\eta)\|\le \alpha$. This proves that $S_{\ell_{\infty}(X_\eta)}\subset C(\alpha,1,\mathbb{R}^+\bigcup\tilde{A}_\eta)$
		\item[$(b)$] If $\alpha\in (-1,0)$, then $\alpha\|x(\eta)\|\le\alpha(1+\varepsilon/\alpha)=\alpha+\varepsilon$. This shows that $S_{\ell_{\infty}(X_\eta)}\subset C(\alpha+\varepsilon,1,\mathbb{R}^+\bigcup\tilde{A}_\eta)$.
	\end{itemize}
\end{proof}

\begin{corollary}\label{cor: infinity sum BCP}
	Let $\kappa$ be an infinite cardinal and $\{X_\eta:\eta<\kappa\}$ a family of normed spaces.
	\begin{itemize}
	    \item[$(a)$] If $\alpha\in[-1,0]$ and each $X_\eta$ satisfies the $\alpha$-BCP$_\kappa$, then $\ell_{\infty}(X_\eta)$ has the $\alpha$-BCP$_{\kappa}$.
	    \item[$(b)$] If $\alpha\in(0,1)$ and each $X_\eta$ satisfies the $\alpha$-BCP$_\kappa$, then, for every $\varepsilon>0$, $\ell_{\infty}(X_\eta)$ has the $(\alpha-\varepsilon)$-BCP$_{\kappa}$.
	\end{itemize}
\end{corollary}

Notice that the converse of Corollary~\ref{cor: infinity sum BCP}, for $\alpha\in [0,1)$, is given by Corollary~\ref{cor: absolute sum}.

\begin{proposition}\label{prop: ell_1 sum}
Let $(X_n)$ be a sequence of normed spaces, for each $n$ let $A_n$ be a subset of $X_n$ and let $\beta\in [-1,1]$.
	\begin{itemize}
		\item[$(a)$] Let $\alpha\in [0,1]$. If $X_n\setminus\{0\}=C(\alpha,\beta,A_n)$ for all $n\in\mathbb{N}$, then $\ell_1(X_n)\setminus\{0\}=C(\alpha+\varepsilon,\beta,c_{00}(A_n\cup\{0\}))$ for all $\varepsilon>0$.
		\item[$(b)$] Let $\alpha\in [-1,1]$. If each $A_n$ is a subspace and there exists some $c>0$ such that $S_{X_n}\subset C(\alpha,\beta,cB_{A_n})$ for all $n\in\mathbb{N}$, then $S_{\ell_1(X_n)}\subset C(\alpha,\beta,cB_{\ell_1(A_n)})$.
	\end{itemize}
\end{proposition}
\begin{proof}
	$(a)$. Fix $x\in\ell_1(X_n)\setminus\{0\}$ and $\varepsilon>0$. Find some $n\in\mathbb{N}$ such that $\sum_{j=n+1}^{\infty}\|x(j)\|\le\varepsilon\|x\|_1$. There are $a_1\in A_1,\ldots,a_n\in A_n$ such that $\|x(j)-a_j\|<\alpha\|x(j)\|+\beta\|a_j\|$ for every $j\le n$ that satisfies $x(j)\not=0$.
	
	Define $a\in c_{00}(A_n\cup\{0\})$ by 
	\[
  a(j) :=
  \begin{cases}
                                   a_j & \text{if $j\leq n$ and $x(j)\neq 0$}, \\
                                   0 & \text{otherwise}. 
  \end{cases}
\]
	Therefore,
	$$\|x-a\|_1=\sum_{j=1}^n\|x(j)-a_j\|+\sum_{j=n+1}^{\infty}\|x(j)\|<(\alpha+\varepsilon)\|x\|_1+\beta\|a\|_1.$$
	
	$(b)$. Fix $x\in S_{\ell_1(X_n)}$. For all $n\in\bb{N}$ with $x(n)\not=0$, we can find $y_n \in c\|x(n)\|\cdot B_{A_n}$ satisfying $\n{x(n)-y_n}< \alpha\n{x(n)} + \beta\n{y_n}$, whose existence is guaranteed by Remark~\ref{rem: properties of coverings}(b). Define $y\in cB_{\ell_1(A_n)}$ by 
		\[
  y(n) :=
  \begin{cases}
                                   y_n & \text{if $n\in\mathbb{N}$ is such that $x(n)\not=0$}, \\
                                   0 & \text{otherwise}. 
  \end{cases}
\]
Thus, $\n{x-y}_1<\alpha\n{x}_1 + \beta\n{y}_1$.
\end{proof}

We end this section by spelling out two useful consequences of the aforementioned results.

\begin{corollary}\label{cor: ell_1 sum}
	Let $(X_n)$ be a sequence of normed spaces, $\kappa$ an infinite cardinal and $\beta\in[-1,1]$.
	\begin{itemize}
		\item [$(a)$] Let $\alpha\in[-1,1]$. If there is a $m\in\bb{N}$ such that $S_{X_m}\notin C(\alpha,\beta,\kappa)$, then $S_{\ell_1(X_n)}\notin C(\alpha,\beta,\kappa)$.
		\item[$(b)$] Let $\alpha\in [0,1]$. If $S_{\ell_1(X_n)}\notin C(\alpha,\beta,\kappa)$, then, for every $\varepsilon>0$, there exists $n\in\bb{N}$ such that $S_{X_n}\notin C(\alpha-\varepsilon,\beta,\kappa)$.
	\end{itemize}
\end{corollary}
\begin{proof}
	$(a)$ follows by Proposition \ref{prop: absolute sum}.
	
	$(b)$ is an easy consequence of Proposition \ref{prop: ell_1 sum}(a).
\end{proof}

\section{Examples and counterexamples of spaces with the $\alpha$-BCP$_\kappa$}\label{sec: examples}

We begin this section by showing that in a big class of Banach spaces, namely in Lipschitz spaces, the notions of having an octahedral norm and failing the $(-1)$-BCP coincide (see Proposition~\ref{prop: OH and failure of (-1)-BCP are the same in Lip(M)}). To achieve this we will first recall the notion of the Daugavet property.

A Banach space $X$ is said to have the \emph{Daugavet property} if every rank-one operator $T\colon X\to X$ satisfies the norm equality
\[
\|I+T\|=1+\|T\|,  
\] 
where $I\colon X\to X$ denotes the identity operator. In 1963, I.~K.~Daugavet (see \cite{Daugavet1963}) was the first to observe that the space $C[0,1]$ has the Daugavet property. By now this property has attracted a lot of attention and many examples have emerged. Indeed, $C(K)$ spaces for a compact Hausdorff space $K$ without isolated points, $L_1(\mu)$ and $L_\infty(\mu)$ spaces for some atomless measure $\mu$, and Lipschitz spaces $\text{Lip}_0(M)$ whenever $M$ is a length space -- all have the Daugavet property (see \cite{garciaCharacterisationDaugavet2018} and \cite{WernerRecentProgress}). 

\begin{lemma}[see {\cite[Lemma~2.12]{KSSW00}}]\label{lem: X Daugavet then X* fails (-1)-BCP}
If $X$ has the Daugavet property, then $X^\ast$ fails the $(-1)$-BCP.
\end{lemma}

Note that in \cite[Theorem~3.5]{Luo2020} it was proved that $L_\infty[0,1]$ fails the BCP. Since $L_1[0,1]$ has the Daugavet property and $(L_1[0,1])^\ast=L_\infty[0,1]$, then, by Lemma~\ref{lem: X Daugavet then X* fails (-1)-BCP}, we have that $L_\infty[0,1]$ even fails the $(-1)$-BCP. Moreover, since $L_\infty[0,1]=\text{Lip}_0([0,1])$, then the previous observation is a particular case of the following result:

\begin{proposition}\label{prop: OH and failure of (-1)-BCP are the same in Lip(M)}
	Let $M$ be a complete metric space. Then the following are equivalent:
	\begin{itemize}
		\item[$(i)$] $\text{Lip}_0(M)$ has the Daugavet property.
		\item[$(ii)$] $\text{Lip}_0(M)$ is octahedral.
		\item[$(iii)$] $\text{Lip}_0(M)$ fails the $(-1)$-BCP.
	\end{itemize}
\end{proposition}
\begin{proof}
The equivalence $(i) \iff (ii)$ follows from \cite[Theorem~1.5]{AMC19}.

$(iii) \implies (ii)$ is clear from definition.

	$(i) \implies (iii)$. Assume that $\text{Lip}_0(M)$ has the Daugavet property. It is known that $\text{Lip}_0(M)$ is a dual space, whose predual is the Lipschitz-free space $\cl F(M)$. Since the Daugavet property passes from the dual space to its predual (see \cite[Lemmata~2.2 and 2.4]{WernerRecentProgress}), then $\cl F(M)$ also has the Daugavet property. Therefore, by Lemma~\ref{lem: X Daugavet then X* fails (-1)-BCP}, $\text{Lip}_0(M)$ fails the $(-1)$-BCP.
\end{proof}

As noted in Remark \ref{rem: rephrase Godefroy's theorem}, Theorem~\ref{thm: godefroy contain ell1} can be expressed also by means of the failure of the $\alpha$-BCP$_{<\omega}$. Nevertheless, it is natural to ask if we can restate Theorem~\ref{thm: renorming ell_1(kappa)} in terms of some infinite version of octahedrality. For this reason we introduce the following definition.

\begin{definition}\label{def: kappa-octahedrality}
	Let $X$ be a normed space and $\kappa$ an infinite cardinal. We say that $X$ is \emph{$\kappa$-octahedral} if for every subspace $Y\subset X$ with $\dens(Y)\le\kappa$ and $\varepsilon>0$, there exists $x\in S_X$ such that for all $\lambda\in\mathbb{R}$ and $y\in Y$ we have $\|\lambda x+y\|\ge(1-\varepsilon)(|\lambda|+\|y\|)$.
\end{definition}

Note that, by Proposition~$\ref{prop: equivalent conditions coverings}$, a normed space $X$ is $\kappa$-octahedral if and only if $S_X\notin C(\alpha,\alpha,\kappa)$ (or equivalently $X\setminus\{0\}\notin C(\alpha,\alpha,\kappa)$) for all $\alpha\in(0,1)$.

We do not know if, up to renorming, the Banach spaces containing $\ell_1(\kappa^+)$ can be characterized in terms of $\kappa$-octahedrality, but we can compare the $\alpha$-BCP$_\kappa$ with $\kappa$-octahedrality. Namely, we finish this section by showing that the following diagram holds.
\vspace{5mm}

\begin{figure}[H]
\caption{} \label{fig: diagram}
\centering
\begin{tikzpicture}
	\path
	(0,3) node (2) {$X$ fails the $(-1)$-BCP$_{\kappa}$}
	(0,1.5) node (3) {$X$ fails the $\alpha$-BCP$_{\kappa}$ for all $\alpha\in (-1,1)$}
	(7,1.5) node (4) {$X$ is $\kappa$-octahedral}
	(0,0) node (5) {$\ell_1(\kappa^+)\subset X$ isomorphically}
	(7,0) node (6) {$X$ is octahedral}
	(-0.9,2.25) node (7) [scale=0.75]{Thm.~\ref{thm: X fails (-alpha)-BCP for all alpha in (0,1)}}
	(4.37,1.95) node (8) [scale=0.75]{Thm.~\ref{thm: OH with BCP}}
	(6.25,0.75) node (9) [scale=0.75]{Ex.~\ref{ex: ell_1 is not w-octahedral}}
	(4,0.45) node (10) [scale=0.75]{Ex.~\ref{ex: ell_1 is not w-octahedral}}
	(4,-0.45) node (11) [scale=0.75]{Ex.~\ref{ex: ell_infty}}
	(-0.75,0.75) node (12) [scale=0.75]{Ex.~\ref{ex: ell_infty}}
	(0.83,0.75) node (13) [scale=0.75]{Prop.~\ref{prop: X contains ell_1}};
	\draw[-implies,double equal sign distance] ([xshift=1mm]2.south)--([xshift=1mm]3.north);
	\draw[nimplies,implies-,double equal sign distance] ([xshift=-1mm]2.south)--([xshift=-1mm]3.north);
	\draw[-implies,double equal sign distance] ([xshift=1mm]3.south)--([xshift=1mm]5.north);
	\draw[nimplies,implies-,double equal sign distance] ([xshift=-1mm]3.south)--([xshift=-1mm]5.north);
	\draw[-implies,double equal sign distance] ([xshift=1mm]4.south)--([xshift=1mm]6.north);
	\draw[nimplies,implies-,double equal sign distance] ([xshift=-1mm]4.south)--([xshift=-1mm]6.north);
	\draw[nimplies,-implies,double equal sign distance] ([yshift=1mm]4.west)--([yshift=1mm]3.east);
	\draw[implies-,double equal sign distance] ([yshift=-1mm]4.west)--([yshift=-1mm]3.east);
	\draw[nimplies,-implies,double equal sign distance] ([yshift=-1mm]5.east)--([yshift=-1mm]6.west);
	\draw[nimplies,implies-,double equal sign distance] ([yshift=1mm]5.east)--([yshift=1mm]6.west);
\end{tikzpicture}
\end{figure}

The unlabelled arrows in Figure~\ref{fig: diagram} follow easily from the corresponding definitions. We begin with some elementary examples.

\begin{example}\label{ex: ell_1 is not w-octahedral}
	The Banach space $\ell_1$ is clearly octahedral, but for every infinite cardinal $\kappa$, it doesn't contain $\ell_1(\kappa^+)$ and it is not ${\omega}$-octahedral. Indeed, $S_{\ell_1}\subset C(\alpha,\alpha,\ell_1)$ for every $\alpha\in(0,1)$, therefore $S_{\ell_1}\in C(\alpha,\alpha,\omega)$ thanks to Remark~\ref{rem: properties of coverings}(a). In general, notice that there is no normed space $X$ which is $\dens(X)$-octahedral.
\end{example}

\begin{example}\label{ex: ell_infty}
	Fix an infinite cardinal $\kappa$. On one hand, by \cite[Fact 7.26]{Hajek_Biorthogonal_2008}, we have that $\ell_1(k^+)\subset\ell_1(2^{\kappa})\subset \ell_{\infty}(\kappa)$. On the other hand, fix any $x\in S_{\ell_\infty(\kappa)}$.
	\begin{itemize}
	    \item If $|x(1)|=1$, then $\|x-\text{sgn}(x(2))e_2\|_\infty=1$.
	    \item If $|x(1)|\not=1$, then $\|x-\text{sgn}(x(1))e_1\|_\infty=1$.
	\end{itemize}
	This shows that $\ell_\infty(\kappa)$ is not octahedral since $S_{\ell_\infty(\kappa)}$ is contained in\\
	$C(2/3,2/3,\{\mp e_1,\mp e_2\})$. In fact for all $x\in S_{\ell_\infty(\kappa)}$ there exists $a\in\{\mp e_1,\mp e_2\}$ such that
	$$\|x-a\|_\infty=1<4/3=2/3\cdot(\|x\|_\infty+\|a\|_\infty).$$
	Moreover, it also shows that $\ell_\infty(\kappa)$ has the $\alpha$-BCP$_{<\omega}$ for all $\alpha\in [-1,0)$ since $S_{\ell_\infty(\kappa)}\subset C(\alpha,1,\{\mp e_1,\mp e_2\})$. In fact we only need to change the previous inequality to
	$$\|x-a\|_\infty=1<\alpha+\|a\|_\infty.$$
\end{example}

Our next aim is to show that there exists a Banach space, which fails the $\alpha$-BCP$_{\kappa}$ for every $\alpha\in(-1,1)$, but that has the $(-1)$-BCP. 

Let us first observe it for $\kappa = \omega$ building on the construction from \cite{cheng-cheng-liu}.

\begin{remark}
The main idea of the necessity proof of \cite[Theorem 2.1]{cheng-cheng-liu} actually gives that the quotient space $\ell_\infty / c_0$ fails the $(-1)$-BCP.
\end{remark}

The remaining part of the same proof can be summed up and extended as follows.
\begin{proposition}
Let $X$ and $Y$ be Banach spaces, let $T \in \cl L(X,Y)$ be onto and let $Y$ fail the $\alpha$-BCP$_{\kappa}$ for some $\alpha \in [-1, 1)$ and some infinite cardinality $\kappa$. Then for every $\e > 0$ there exists $\lambda > 0$ such that $(X, \norm{\cdot}_\lambda)$ fails
the $(\alpha+\e)$-BCP$_{\kappa}$, where 
$\norm{x}_\lambda = \lambda \norm{x} + \n{Tx}$ is an equivalent norm on $X$.
\end{proposition}
\begin{proof}
By the open mapping theorem there exists $r >0$ such that $rB_Y \subset T(B_X)$.
Take a set $A \subset X$ with $|A| \leq \kappa$. Find an element $y \in Y$ such that $\n{y - Ta} \geq \n{Ta} - \alpha \n y$ for all $a \in A$ and find $x$ such that $Tx = y$ and $r \n x \leq \n{Tx}$. Note that $\n x_\lambda \geq (\lambda + r)  \n x \geq r \n x$.
Then
\begin{align*}
\n{x - a}_\lambda &= \lambda \n{x-a} + \n{Tx-Ta}\\
&\geq \lambda(\n a - \n x) + \n{Ta} - \alpha \n{Tx}\\
&= \n{a}_\lambda - \alpha \n{x}_\lambda - \lambda(1- \alpha) \n{x}\\
& \geq \n{a}_\lambda - \alpha \n{x}_\lambda - \frac{\lambda(1 - \alpha)}{r} \n{x}_\lambda,
\end{align*}
so that $(X, \n{\cdot}_\lambda)$ fails the $\left(\alpha + \frac{\lambda(1 - \alpha)}{r}\right)$-BCP$_\kappa$.
\end{proof}

\begin{remark}
\label{neg bcp gap}
Note that in the proposition above, a set  $A \subset \ker T$ satisfies  $S_X \subset C(1,1,A)$ in the new norm if and only if it does so in the original norm.
\end{remark}

\begin{example}
    
Applying the above to the quotient map $T: \ell_\infty \to \ell_\infty / c_0$, we get, for any $\e >0$, a renorming $X_\e$ of $\ell_\infty$, failing the $(-1 + \e)$-BCP, but satisfying $S_{X_\e} \subset C(1,1, 2B_{c_0})$ (in particular, enjoying the $(-1)$-BCP).

Consider the space $\ell_1(X_{1/n})$,
\cite[Proposition~7]{Guirao2019} 
(or its more general version, Corollary ~\ref{cor: ell_1 sum}.(a)) say that it fails the $(-1+\e)$-BCP for all $\e > 0$. On the other hand,
$\ell_1(X_{1/n})$  has the $(-1)$-BCP by
Proposition~\ref{prop: ell_1 sum}.(b).
\end{example}

For the general case, when $\kappa$ is not necessarily $\omega$, let us start with a different building block.

\begin{lemma}\label{lem: renorm ell_1}
	Let $\kappa$ be an infinite cardinal and $\alpha\in (-1,0)$. Let $X_{\alpha,\kappa}:=(\ell_1(\kappa^+),\|\cdot\|_{\alpha})$, where 
	$$\|\cdot\|_{\alpha}:=-\alpha\|\cdot\|_1+(1+\alpha)\|\cdot\|_{\infty}.
	$$
	Then $X_{\alpha,\kappa}$ fails the $\alpha$-BCP$_{\kappa}$, but $S_{X_{\alpha,\kappa}}\subset C(1,1,\{\mp e_1\})$.
\end{lemma}
\begin{proof}
	Fix a set $A\subset X_{\alpha,\kappa}$ of cardinality $\kappa$ and find $\theta\in\kappa^+\setminus\bigcup_{a\in A}\text{supp}(a)$. Then for all $a\in A$
	$$\|a-e_{\theta}\|_{\alpha}=-\alpha(\|a\|_1+1)+(1+\alpha)\max\{\|a\|_{\infty},1\}\ge$$
	$$\ge-\alpha(\|a\|_1+1)+(1+\alpha)\|a\|_{\infty}=\|a\|_{\alpha}-\alpha.$$
	This proves that $X_{\alpha,\kappa}$ fails the $\alpha$-BCP$_{\kappa}$. 
	
	Finally, we show that $S_{X_{\alpha,\kappa}}\subset C(1,1,\{\mp e_1\})$. Let $x\in S_{X_{\alpha,\kappa}}$. If $x(1)=0$, then $\|x+e_1\|_{\infty}=1$. Otherwise, say $x(1)<0$, note that $\|x+e_1\|_1<2$ (if $x(1)>0$, then use $-e_1$ instead). In either case we conclude that for all $x\in X_{\alpha,\kappa}$ there exists $e\in\{\mp e_1\}$ such that $\|x-e\|_\alpha<2$.
\end{proof}

\begin{theorem}\label{thm: X fails (-alpha)-BCP for all alpha in (0,1)}
	For every infinite cardinal $\kappa$ there exists a Banach space $X$ which fails the $\alpha$-BCP$_{\kappa}$ for every $\alpha\in(-1,1)$, but that has the $(-1)$-BCP.
\end{theorem}
\begin{proof}
	Let $\kappa$ be an infinite cardinal. Fix a sequence $(\alpha_n)\subset (-1,0)$ such that $\inf_n\alpha_n=-1$. For every $n\in\mathbb{N}$ define a Banach space $X_n:=X_{\alpha_n,\kappa}$ as in Lemma~\ref{lem: renorm ell_1} and set $X:=\ell_1(X_n)$. Now, Corollary~\ref{cor: ell_1 sum}(a) implies that $X$ fails the $\alpha$-BCP$_{\kappa}$ for every $\alpha\in(-1,1)$. Moreover, by Proposition \ref{prop: ell_1 sum}(b), $X$ has the $(-1)$-BCP.
\end{proof}

We end this section by proving the existence of a Banach space which is $\kappa$-octahedral, but that has the $\alpha$-BCP for all $\alpha\in [-1,0)$.

\begin{lemma}\label{lem: ell_p}
	For all $\alpha\in(1/2,1)$ and infinite cardinal $\kappa$ there exists a Banach space $X_{\alpha,\kappa}$ that has the $\beta$-BCP for all $\beta\in[-1,0)$ and such that $S_{X_{\alpha,\kappa}}\notin C(\alpha,\alpha,\kappa)$.
\end{lemma}
\begin{proof}
	Set $p:=(1+\log_2(\alpha))^{-1}\in(1,\infty)$ and $X_{\alpha,\kappa}:=\ell_p(\kappa^+)$. Fix a set $A\subset X_{\alpha,\kappa}$ of cardinality $\kappa$. Find some $\theta\in\kappa^+\setminus\bigcup_{a\in A}\text{supp}(a)$. For all $a\in A$, by Jensen's inequality, we have that
	$$\|e_{\theta}-a\|_p=\bigg(1+\sum|a(\eta)|^p\bigg)^{\frac{1}{p}}\ge\frac{1}{2}2^{\frac{1}{p}}+\frac{1}{2}\bigg(2\sum|a(\eta)|^p\bigg)^{\frac{1}{p}}=$$
	$$=2^{\frac{1}{p}-1}(1+\|a\|_p)=\alpha(1+\|a\|_p).$$
	This proves that $S_{X_{\alpha,\kappa}}\notin C(\alpha,\alpha,A)$. On the other hand $X_{\alpha,\kappa}$ is reflexive, hence it does not contain $\ell_1$. Therefore, by Proposition~\ref{prop: X contains ell_1}, $X_{\alpha,\kappa}$ has the $\beta$-BCP for all $\beta\in[-1,0)$.
\end{proof}

\begin{theorem}\label{thm: OH with BCP}
	For every infinite cardinal $\kappa$ there exists a Banach space which is $\kappa$-octahedral, but that has the $\alpha$-BCP for all $\alpha\in[-1,0)$.
\end{theorem}
\begin{proof}
	Fix a sequence $(\alpha_n)\subset (1/2,1)$ such that $\sup_n\alpha_n=1$ and for all $n\in\mathbb{N}$ pick a Banach space $X_n:=X_{\alpha_n,\kappa}$ as in Lemma~\ref{lem: ell_p}. Now, Corollary~\ref{cor: ell_1 sum}(b) implies that $X:=\ell_1(X_n)$ has the $\alpha$-BCP for all $\alpha\in[-1,0)$. Moreover, Corollary~\ref{cor: ell_1 sum}(a) shows that $X$ is $\kappa$-octahedral.
\end{proof}

\section{Remarks and open questions}\label{sec: questions}

Let us end the paper with some questions that are suggested by the current work.

Recall that Theorem~\ref{thm: renorming ell_1(kappa)} shows that if a Banach space contains $\ell_1(\kappa^+)$, then it can be equivalently renormed to fail the $\alpha$-BCP$_\kappa$ for all $\alpha\in (-1,1)$. Therefore, we wonder whether there is a renorming such that the space fails even the $(-1)$-BCP$_\kappa$. Note that Theorem~\ref{thm: X fails (-alpha)-BCP for all alpha in (0,1)} suggests that the answer to the following question could be negative. 

\begin{ques}
Let $\kappa$ be an infinite cardinal. Can every Banach space containing an isomorphic copy of $\ell_1(\kappa^+)$ be equivalently renormed so that it fails the $(-1)$-BCP$_{\kappa}$?
\end{ques}

As already mentioned in Section 5, it is natural to ask whether we can rephrase Theorem~\ref{thm: renorming ell_1(kappa)} in terms of $\kappa$-octahedrality in order to obtain a result closer to Theorem~\ref{thm: godefroy contain ell1}. Notice that Proposition~\ref{prop: renorm X to fail -alpha-BCP} provides one direction since the failure of the $\alpha$-BCP$_\kappa$ for all $\alpha\in (-1,1)$ implies $\kappa$-octahedrality, but Theorem~\ref{thm: OH with BCP} suggests that the converse might not hold. 

\begin{ques}
Let $\kappa$ be an infinite cardinal. If a Banach space is ${\kappa}$-octahedral, then does it contain an isomorphic copy of $\ell_1(\kappa^+)$?
\end{ques}

In particular, we do not know the answers to these questions even in the case $\kappa = \omega$.

\bibliographystyle{siam}

\end{document}